\documentclass{amsart}

\usepackage{amsthm,amsmath}
\usepackage{amssymb, amsfonts}
\usepackage{chngcntr}
\usepackage{color}
\usepackage{dsfont}
\usepackage{enumitem}
\usepackage{fancyhdr}
\usepackage{graphicx}
\usepackage{hyperref}
\usepackage{mathtools}
\usepackage{tikz}
\usepackage{tikz-cd}
\usepackage{ulem}
\usepackage{upgreek}
\usepackage{wasysym}
\usepackage{lipsum}

\usepackage[font=footnotesize]{caption}

\usepackage{cleveref}

\usetikzlibrary{trees}

\newtheorem{thm}{Theorem}[section]
\newtheorem{theorem}[thm]{Theorem}

\newtheorem{lemma}[thm]{Lemma}

\theoremstyle{definition}

\newtheorem{notation}[thm]{Notation}

\newtheorem{fact}[thm]{Fact}
\newtheorem{definition}[thm]{Definition}

\newtheorem{rem/def}[thm]{Remark/Definition}

\counterwithin{figure}{thm}

\newtheorem*{thmx}{Theorem \ref{thm: criterion NCTP}}
\newtheorem*{thmy}{Theorem \ref{thm: criterion NBTP}}
\newtheorem*{lemmaz}{Lemma \ref{lem: weak BTP}}

\def \lex {<_{lex}}

\newcommand{\trn}{%
  \mathrel{\ooalign{$\lneq$\cr\raise.21ex\hbox{$\lhd$}\cr}}}
\newcommand{\trrn}{%
  \mathrel{\ooalign{$\gneq$\cr\raise.21ex\hbox{$\rhd$}\cr}}}
\def \coc {{^{\frown}}}

\def \cl {{\rm cl}}

\def \la {\langle}
\def \ra {\rangle}
\def \Tau {\mathcal{T}}

\def \age {\textrm{age}}

\def \CI {\mathcal{I}}
\def \CJ {\mathcal{J}}
\def \CL {\mathcal{L}}

\def \CT {\mathcal{T}}

\def \cf {\text{cf}}
\def\tp{\operatorname{tp}}
\def\Th{\operatorname{Th}}

\def\qftp{\operatorname{qftp}}

\title{Mekler's Construction and the Preservation of NCTP and NBTP}
\author[JinHoo Ahn]{JinHoo Ahn}
\address{Changwon National University, Department of Mathematics\\ 20 Changwondaehakro, Uichang-gu\\ 51140, Changwon-si, Gyeongsangnam-do, South Korea,}
\email{jinu1229@changwon.ac.kr}
\author[Joonhee Kim]{Joonhee Kim}
\address{Korea Institute for Advanced Study, School of Mathematics\\ 85 Hoegiro Dongdaemun-gu\\ 02455, Seoul, South Korea}
\email{kimjoonhee@kias.re.kr}
\thanks{The first author is supported by the Global - Learning \& Academic research institution for Master's $\cdot$ PhD students, and Postdocs (LAMP) Program of the National Research Foundation of Korea (NRF) grant funded by the Ministry of Education (No.~RS-2024-00444460), and by the National Research Foundation of Korea (NRF) grant funded by the Korea government (MSIT) (No.~RS-2024-00451678).
The second author is supported by a KIAS individual grant (project no.~6G091801) and by an NRF of Korea grant (No.~2021R1A2C1009639).}

\begin{document}

\begin{abstract}
We give criteria for a first-order theory to be NCTP or NBTP using tree-indiscernibility.
As an application, we show that Mekler’s construction preserves NCTP and NBTP.
\end{abstract}

\maketitle

\section{Introduction}

Mekler’s construction \cite{Mek81} provides a uniform method for interpreting an arbitrary first-order structure inside a 2-nilpotent group of finite exponent. Although this interpretation is not a bi-interpretation, it has proved remarkably robust in preserving a wide range of model-theoretic dividing lines arising in Shelah’s classification theory \cite{She90}. In his original work \cite{Mek81}, Mekler introduced this construction and established several model-theoretic features of it, including the preservation of stability. Building on this, subsequent results showed that the construction preserves simplicity and, assuming stability, CM-triviality \cite{Bau02}. It was later established that further dividing lines such as NIP, $k$-dependence, and NTP$_2$ are preserved \cite{CH18}, as well as NTP$_1$ \cite{Ahn19}. More recently, it has now been shown that Mekler’s construction preserves every dividing line defined via collapsing indiscernibles \cite{BPT24}.

On the other hand, over the past few years, several attempts have been made to study new dividing lines that encompass both the classes of NTP$_1$ and NTP$_2$ theories, namely NATP, NCTP, and NBTP. These notions are defined via consistency–inconsistency configurations indexed by tree-structures, and various properties of these dividing lines have been gradually investigated in subsequent work \cite{AK20, AKL21, AKLL22, Han23, Han25, KL23, KR23, M-G25, Mut22}. Their relationship to other dividing lines is described by the following implications.
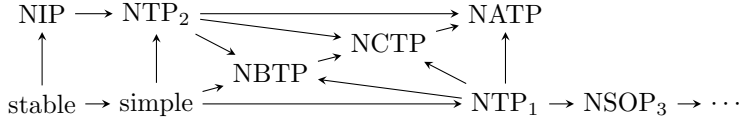
\[
\begin{tikzpicture}[>=stealth, x=1.6cm, y=1.2cm]
\node (stable) at (0.15,-1) {stable};
\node (NIP) at (0.15,0) {NIP};
\node (NTP)   at (1.1,-1) {simple};
\node (NTP2)   at (1.1,0) {NTP$_2$};
\node (NTP1)   at (4,-1) {NTP$_1$};
\node (NSOP3)   at (5,-1) {NSOP$_3$};
\node (cdots)   at (5.85,-1) {$\cdots$};
\node (NBTP)   at (2.05,-0.66) {NBTP};
\node (NCTP)   at (3.03,-0.33) {NCTP};
\node (NATP)   at (4,0) {NATP};

\draw[->] (stable) -- (NIP);
\draw[->] (stable) -- (NTP);
\draw[->] (NTP) -- (NTP2);
\draw[->] (NTP) -- (NTP1);
\draw[->] (NIP) -- (NTP2);
\draw[->] (NTP1) -- (NSOP3);
\draw[->] (NSOP3) -- (cdots);
\draw[->] (NTP) -- (NBTP);
\draw[->] (NTP1) -- (NBTP);
\draw[->] (NTP2) -- (NBTP);
\draw[->] (NTP1) -- (NCTP);
\draw[->] (NTP2) -- (NCTP);
\draw[->] (NTP1) -- (NATP);
\draw[->] (NTP2) -- (NATP);
\draw[->] (NBTP) -- (NCTP);
\draw[->] (NCTP) -- (NATP);
\end{tikzpicture}
\]

The preservation of NATP under Mekler’s construction was established in earlier joint work of the authors with Lee \cite{AKL21}. However, despite the success of Mekler’s construction in preserving many major dividing lines, the behavior of NCTP and NBTP under this construction has not previously been addressed. The main goal of this paper is to show that Mekler’s construction also preserves these two dividing lines. We show that if a theory T is NCTP or NBTP, then the theory of its Mekler group is again NCTP or NBTP, respectively.

A key technique in proving the main results is the development of new criteria for NCTP and NBTP formulated in terms of tree indiscernibility, as below.
\begin{thmx}
The following are equivalent.
\begin{itemize}
\item[(i)] $T$ is NCTP.
\item[(ii)] For any parameter $b$, a cardinal $\kappa>2^{|T|+|b|+\aleph_0}$ with $\cf(\kappa)=\kappa$, and a strong indiscernible tree $(a_\eta)_{\eta\in\omega^{<\kappa}}$, there exist $b'$ and $i<\kappa$  such that 
\begin{itemize}
\item[$\ast$] $b'\equiv_{a_{0^i{}^\frown(1)}}\!b$,
\item[$\ast$] $(a_{0^i{}^\frown(1)^\frown 0^j})_{j<\kappa}$ is order indiscernible over $b'$.
\end{itemize}
\end{itemize}
\end{thmx}

\begin{thmy}
The following are equivalent.
\begin{itemize}
\item[(i)] $T$ is NBTP.
\item[(ii)] For any parameter $b$, a cardinal $\kappa>2^{|T|+|b|+|\aleph_0|}$ with $\cf(\kappa)=\kappa$, and a level-s-indiscernible tree $(a_\eta)_{\eta\in\omega^{<\kappa}}$, there exist $i<\kappa$ and $b'$ such that 
\begin{itemize}
\item[$\ast$] $b'\equiv_{a_{0^i{}^\frown(1)}}\!b$,
\item[$\ast$] $(a_{0^i{}^\frown(j+1)})_{j<\omega}$ or $(a_{0^i{}^\frown(2)^j{}^\frown(1)})_{j<\kappa}$ is order indiscernible over $b'$.
\end{itemize}
\end{itemize}
\end{thmy}
\noindent In the proof of \Cref{thm: criterion NBTP}, we use the following slightly more refined formulation of BTP.
\begin{lemmaz}
$T$ has BTP if and only if there exist $\varphi(x,y)$, $(a_\eta)_{\eta\in\omega^{<\omega}}$, and $k<\omega$ such that
\begin{itemize}
\item[(i)] $\{\varphi(x,a_\eta)\}_{\eta\in X}$ is consistent for any strict left-leaning path $X$,
\item[(ii)] $\{\varphi(x,a_\eta)\}_{\eta\in X}$ is $k$-inconsistent for any strict right-veering path $X$,
\item[(iii)] $\{\varphi(x,a_\eta)\}_{\eta\in X}$ is $k$-inconsistent for any direct sibling set $X$.
\end{itemize}
\end{lemmaz}

Characterizations analogous to the above theorems for NTP$_1$, NTP$_2$, and NATP were established in \cite[Proposition 3.7]{Ahn19}, \cite[Fact 5.3]{CH18}, and \cite[Theorem 3.25]{AKL21}, respectively. As those results have been useful in the study of these dividing lines, we expect that Theorems~\ref{thm: criterion NCTP} and~\ref{thm: criterion NBTP} will likewise be useful for the investigation of NCTP and NBTP theories. The results on Mekler’s construction presented in this paper provide one such application.

\section{Preliminary}
We assume familiarity with basic model-theoretic notions. Below we recall the notation and terminology that will be used frequently throughout the paper.

\subsection{Tree language and tree properties}

\begin{notation}\label{notation: language of trees}
Let $\kappa$ and $\lambda$ be cardinals.
\begin{enumerate}
\item[(i)] $\kappa^\lambda$ is the set of all functions from $\lambda$ to $\kappa$. 
\item[(ii)] $\kappa^{<\lambda}:=\bigcup_{\alpha<\lambda}{\kappa^\alpha}$.
\item[(iii)] By $\emptyset$, we mean the {\it empty string} in $\kappa^{<\lambda}$, which means the empty set (recall that every function can be regarded as a set of ordered pairs). Thus $\emptyset \in \kappa^0$.
\item[(iv)] By $(i_0...i_{n-1})$ for $i_0,...,i_{n-1}<\kappa$, we mean the function $\eta\in\kappa^n$ such that $\eta(m)=i_m$ for each $m<n$.
\item[(v)] By $i^\alpha$ and $(i)^\alpha$ for $i<\kappa$ and $\alpha<\lambda$, we mean the function $\eta\in\kappa^\alpha$ such that $\eta(m)=i$ for all $m<\alpha$. Note that $i^0=(i)^0=\emptyset$.
\end{enumerate}
\smallskip
Let $\eta,\nu\in \kappa^{<\lambda}$, $i<\kappa$, $\alpha<\lambda$.
\begin{enumerate}
\item[(vi)] By $\eta\unlhd\nu$, we mean $\eta \subseteq \nu$. So $\kappa^{<\lambda}$ is partially ordered by $\unlhd$. If $\eta\unlhd\nu$ or $\nu\unlhd\eta$, then we say $\eta$ and $\nu$ are {\it comparable}. 

\item[(vii)] By $\eta\perp\nu$, we mean that $\eta\not\!\!\unlhd\,\nu$ and $\nu\not\!\!\unlhd\,\eta$. We say $\eta$ and $\nu$ are {\it incomparable} if $\eta\perp\nu$.
\item[(viii)] By $\eta\wedge\nu$, we mean the $\unlhd$-maximal $\xi\in\kappa^{<\lambda}$ such that $\xi\unlhd\eta$ and $\xi\unlhd\nu$.
\item[(ix)] By $l(\eta)$, we mean the domain of $\eta$.
\item[(x)] By $\eta\lex\nu$, we mean that either $\eta\lhd\nu$, or $\eta\perp\nu$ and $\eta(l(\eta\wedge\nu))<\nu(l(\eta\wedge\nu))$. 
\item[(xi)] By $\eta\coc\nu$, we mean $\eta\cup\{(l(\eta)+i,\nu(i)):i< l(\nu)\}$. Note that $\emptyset\coc\nu$ is just $\nu$.
\item[(xii)] By $\eta\lhd_i\nu$, we mean $\eta^\frown(i)\unlhd\nu$.
\item[(xiii)] If $l(\eta)$ is a successor ordinal, then $t(\eta):=\eta(l(\eta)-1)$ and $\eta^-\!:=\eta|_{l(\eta)-1}$.
\item[(xiv)] $C(\eta):=\{\eta'\in \kappa^{<\lambda}: \eta'\unrhd\eta\}$.
\item[(xv)] $P_\alpha:=\{\eta\in\kappa^{<\lambda}: l(\eta)=\alpha\}$.
\end{enumerate}
\smallskip
Let $X\subseteq \kappa^{<\lambda}$.
\begin{enumerate}
\item[(xvi)]
By $\eta\coc X$ and $X\coc\eta$, we mean $\{\eta\coc x:x\in X\}$ and $\{x\coc\eta:x\in X\}$ respectively.
\end{enumerate}
\smallskip
Let $\eta_0,...,\eta_n\in\kappa^{<\lambda}$.
\begin{enumerate}
\item[(xvii)] By $\cl(\eta_0,...,\eta_n)$, we mean a tuple $(\eta_0\wedge\eta_0,...,\eta_0\wedge\eta_n,...,\eta_n\wedge\eta_0,...,\eta_n\wedge\eta_n)$. 
\end{enumerate}
\end{notation}

Recently, several new dividing lines have been introduced that encompass both NTP$_1$ and NTP$_2$, and their basic properties have gradually started to be investigated.

\begin{definition}
Let $X\subseteq \omega^{<\omega}$.
\begin{itemize}
\item[(i)] $X$ is called an {\it path} if $X$ is linearly ordered by $\unlhd$.
\item[(ii)] $X$ is called an {\it antichain} if $\eta\perp\nu$ for all $\eta,\nu\in X$.
\item[(iii)] $X$ is called a {\it descending comb} if it is an antichain and there is some enumeration $\{\eta_i\}_{i<\kappa}$ of $X$ such that $\eta_{i+1}\lex \eta_i$ and $\bigwedge_{i\le j}\eta_j\lhd \bigwedge_{i+1\le j}\eta_j $ for all $i<\kappa$.
\item[(iv)] $X$ is called a {\it left-leaning path} if $\emptyset\notin X$ and there is some enumeration $\{\eta_i\}_{i<\kappa}$ of $X$ such that $\eta_i^- {^\frown} (j)\lhd \eta_{i+1}$ for some $j\le t(\eta_i)$ for all $i<\kappa$. 
\item[(v)] $X$ is called a {\it right-veering path} if $\emptyset\notin X$ and there is some enumeration $\{\eta_i\}_{i<\kappa}$ of $X$ such that $\eta_i^- {^\frown} (j)\unlhd \eta_{i+1}$ for some $j> t(\eta_i)$ for all $i<\kappa$. 
\end{itemize}
\end{definition}

\begin{definition}\label{def: ATP CTP BTP} 
Let $T$ be a complete theory.
\begin{itemize}
\item[(i)] \cite{AK20} We say a theory $T$ has {\it ATP} if there exist $\varphi(x,y)$ and $(a_\eta)_{\eta\in\omega^{<\omega}}$ such that
\begin{itemize}
\item[$\ast$] $\{\varphi(x,a_\eta)\}_{\eta\in X}$ is consistent for any antichain $X$ in $\omega^{<\omega}$
\item[$\ast$] $\{\varphi(x,a_\eta),\varphi(x,a_\nu)\}$ is inconsistent for any $\eta,\nu\in\omega^{<\omega}$ with $\eta\lhd\nu$.
\end{itemize}

\smallskip

\noindent If $T$ does not have ATP, then we say $T$ is {\it NATP}.

\smallskip

\item[(ii)] \cite{Mut22}\footnote{In the original definition by Mutchnik \cite{Mut22}, it is named DCTP$_2$, and the underlying index set is $2^{<\omega}$. It is easy to check that having DCTP$_2$ and having CTP are equivalent.} We say a theory $T$ has {\it CTP} if there exist $\varphi(x,y)$, $(a_\eta)_{\eta\in\omega^{<\omega}}$, and $k<\omega$ such that
\begin{itemize}
\item[$\ast$] $\{\varphi(x,a_\eta)\}_{\eta\in X}$ is consistent for any descending comb $X$ in $\omega^{<\omega}$,
\item[$\ast$] $\{\varphi(x,a_{\eta|_n})\}_{n<\omega}$ is $k$-inconsistent for any $\eta\in\omega^\omega$.
\end{itemize}

\smallskip

\noindent If $T$ does not have CTP, then we say $T$ is {\it NCTP}.

\smallskip

\item[(iii)] \cite{KR23}\footnote{In the original definition by Kruckman and Ramsey \cite{KR23}, the underlying index set is $\omega^{<\omega}\setminus{\emptyset}$. For technical convenience, we will instead use a formulation indexed by $\omega^{<\omega}$. It is not difficult to check that this version is equivalent to the original one.} We say a theory $T$ is {\it BTP} if there exist $\varphi(x,y)$, $(a_\eta)_{\eta\in\omega^{<\omega}}$, and $k<\omega$ such that
\begin{itemize}
\item[$\ast$] $\{\varphi(x,a_\eta)\}_{\eta\in X}$ is consistent for any left-leaning $X$ in $\omega^{<\omega}$,
\item[$\ast$] $\{\varphi(x,a_\eta)\}_{\eta\in X}$ is $k$-inconsistent for any right-veering $X$ in $\omega^{<\omega}$.
\end{itemize}

\smallskip

\noindent If $T$ does not have BTP, then we say $T$ is {\it NBTP}.
\end{itemize}
\end{definition}

The relationships between NATP and NCTP, and between NBTP and the other dividing lines, are known as follows. Their relationships with stability, simplicity, NIP, and NSOP$_n$ for $n>2$ are illustrated in \Cref{fig: relations between dividing lines}.

\begin{fact}
Let $T$ be a complete theory.
\begin{itemize}
\item[(i)] \cite{KR23} If $T$ is NTP$_1$ or NTP$_2$, then it is NBTP.
\item[(ii)] \cite{Han23}\cite{KR23} If $T$ is NBTP, then it is NCTP.
\item[(iii)] \cite{Mut22} If $T$ is NCTP, then it is NATP.
\end{itemize}
\end{fact}

\begin{figure}
\[
\begin{tikzpicture}[>=stealth, x=1.6cm, y=1.2cm]
\node (stable) at (0.15,-1) {stable};
\node (NIP) at (0.15,0) {NIP};
\node (NTP)   at (1.1,-1) {simple};
\node (NTP2)   at (1.1,0) {NTP$_2$};
\node (NTP1)   at (4,-1) {NTP$_1$};
\node (NSOP3)   at (5,-1) {NSOP$_3$};
\node (cdots)   at (5.85,-1) {$\cdots$};
\node (NBTP)   at (2.05,-0.66) {NBTP};
\node (NCTP)   at (3.03,-0.33) {NCTP};
\node (NATP)   at (4,0) {NATP};

\draw[->] (stable) -- (NIP);
\draw[->] (stable) -- (NTP);
\draw[->] (NTP) -- (NTP2);
\draw[->] (NTP) -- (NTP1);
\draw[->] (NIP) -- (NTP2);
\draw[->] (NTP1) -- (NSOP3);
\draw[->] (NSOP3) -- (cdots);
\draw[->] (NTP) -- (NBTP);
\draw[->] (NTP1) -- (NBTP);
\draw[->] (NTP2) -- (NBTP);
\draw[->] (NTP1) -- (NCTP);
\draw[->] (NTP2) -- (NCTP);
\draw[->] (NTP1) -- (NATP);
\draw[->] (NTP2) -- (NATP);
\draw[->] (NBTP) -- (NCTP);
\draw[->] (NCTP) -- (NATP);
\end{tikzpicture}
\]
\caption{Relations between dividing lines}
\label{fig: relations between dividing lines}
\end{figure}
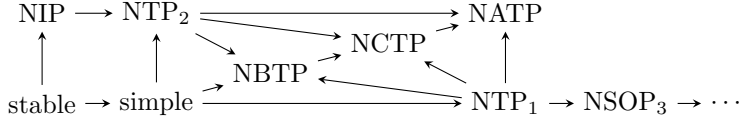

\subsection{Generalized indiscernibles}

\begin{definition}
Let $\CL^*$ be a language, $\CI^*$ an $\CL^*$-structure. Let $\CL$ be another language, $T$ an $\CL$-theory, and $\mathbb{M}$ a sufficiently saturated model of $T$. Fix a cardinal $\kappa$, a small subset $A$ of $\mathbb{M}$, and an $\CI^*$-indexed set $(a_i)_{i\in\CI^*}$ such that $a_i\in \mathbb{M}^\kappa$ for each $i\in\CI^*$. We say $(a_i)_{i\in\CI^*}$ is {\it $\CI^*$-indiscernible over $A$} if 
\[ \qftp_{\CL^*}^{\CI^*}(I)=\qftp_{\CL^*}^{\CI^*}(J)\;\Rightarrow\; \tp_\CL^\mathbb{M}( (a_i)_{i\in I}/A )=\tp_\CL^\mathbb{M}( (a_i)_{i\in J}/A ) \]
for all $I,J\subseteq\CI^*$.
\end{definition}

\begin{definition}\cite[Definition 3.1.1.]{Sco10}\label{def:generalized indiscernible}
Let $\CL^*$ and $\CI^*$ be as above. We say $\CI^*$-indiscernibles have the {\it modeling property} if for any language $\CL$, an $\CL$-theory $T$, a sufficiently saturated model $\mathbb{M}\models T$, a small subset $A$ of $\mathbb{M}$, and an $\CI^*$-indexed set $(a_i)_{i\in\CI^*}\subseteq\mathbb{M}$ with $|a_i|=|a_j|$ for all $i,j\in\CI^*$, we can always find $(b_i)_{i\in\CI^*}$ in $\mathbb{M}$ such that 
\begin{itemize}
\item[(i)] $(b_i)_{i\in\CI^*}$ is $\CI^*$-indiscernible over $A$,
\item[(ii)] for any $I\subseteq\CI^*$ and $\varphi(x)\in \tp_\CL^\mathbb{M}(  (b_i)_{i\in I} /A)$, there exists $J\subseteq\CI^*$ such that $\qftp_{\CL^*}^{\CI^*}(I)=\qftp_{\CL^*}^{\CI^*}(J)$ and $\mathbb{M}\models \varphi((a_i)_{i\in J})$.
\end{itemize}

If some $\CI^*$-indexed set satisfies (ii), then we say the $\CI^*$-indexed set is {\it locally based} on $(a_i)_{i\in\CI^*}$.
\end{definition}

\begin{notation}
$\CL_0 := \{\unlhd, \lex, \wedge\}$. We can regard $\omega^{<\omega}$ as an $\CL_0$-structure by giving interpretations of the symbols of $\CL_0$ as in \Cref{notation: language of trees}. We will write a subscript $\CL_0$ on $\omega^{<\omega}$, namely $\omega^{<\omega}_{\CL_0}$, when we want to emphasize the language that the structure is based on. We say that $(a_\eta)_{\eta \in \omega^{<\omega}}$ is {\it strongly indiscernible} if it is $\omega^{<\omega}_{\CL_0}$-indiscernible. 

Similarly, we can give an $\CL_s:=\{\unlhd,\lex,\wedge\}\cup\{P_i\}_{i<\omega}$-structure on $\omega^{<\omega}$. We say that $(a_\eta)_{\eta \in \omega^{<\omega}}$ is {\it s-indiscernible} if it is $\omega^{<\omega}_{\CL_s}$-indiscernible.
\end{notation}

\begin{fact}\cite{TT12}\cite{KKS14}\label{fact: s-indisc}
Strong indiscernible and s-indiscernible have the modeling property.
\end{fact}

\begin{notation}
Let $\CL^*$ be a language, $\CI^*$ an $\CL^*$-structure, and $I$ its underlying set. $\age_{\CL^*}(\CI^*)$ is the set of all $\CL^*$-substructures of $\CI^*$ which are generated by finite subsets of $I$. If it is clear from the context, then we omit the subscript $\CL^*$ and just write $\age(\CI^*)$.
\end{notation}

\begin{definition}
For a language $\CL^*$, an $\CL^*$-structure $\CI^*$ is said to be {\it locally finite} if every finitely generated substructure of $\CI^*$ is  finite.
\end{definition}

\begin{fact}\cite{MP23}\label{fact: age(I)=age(J)}
Let $\CL^*$ be a language, $\CI^*$ and $\CJ^*$ be locally finite $\CL^*$-structures with $\age(\CI^*)=\age(\CJ^*)$. If $\CI^*$-indiscernible has the modeling property, then $\CJ^*$-indiscernible also has the modeling property.
\end{fact}

\subsection{Mekler's construction}

In this subsection, we summarize the basic facts concerning Mekler groups, which are the main objects of study in this paper. Our exposition largely follows the treatment in \cite{Hod93}.

\begin{definition}
A graph is said to be {\it nice} if 
\begin{itemize}
\item[(i)] it has at least two elements,
\item[(ii)] for any two distinct vertices $a$ and $b$, there exists vertex $c$ such that $c\neq a$, $c\neq b$, $c$ is adjacent to $a$, and $c$ is not adjacent to $b$,
\item[(iii)] the graph has no cycle of order $3$ nor $4$.
\end{itemize}
\end{definition}

\begin{fact}\cite[Theorem 5.5.1]{Hod93} 
Every countably infinite structure $M$ in a countable language can be bi-interpreted in a countably infinite nice graph.
\end{fact}

\begin{definition}\cite{Mek81},\cite[Exercise 9.23 and Appendix 3]{Hod93}
Fix an odd prime $p$. For a countably infinite nice graph $C:=\{c_i\}_{i<\omega}$, let $F_p(C):=(X\times Y,\ast)/W$ where $X:=\bigoplus_{i<\omega}x_i(\mathbb{Z}/p\mathbb{Z})$, $Y:=\bigoplus_{i<j<\omega}y_{i,j}(\mathbb{Z}/p\mathbb{Z})$ for indeterminants $x_i$ and $y_{i,j}$, the group operation $\ast$ is defined by 
\[
\left(\sum_{i<\omega}\alpha_i x_i, c\right)\ast \left(\sum_{i<\omega}\beta_i x_i, d\right):=\left(\sum_{i<\omega} (\alpha_i+\beta_i)x_i, c+d-\!\!\! \sum_{i<j<\omega}\!\! \beta_i\alpha_j y_{i,j}\right)
\]
for integers $\alpha_i$, $\beta_i$ such that $0\le \alpha_i,\beta_i<p$, and $W:=\la \{[(x_i,0),(x_j,0)]\ast (0,y_{i,j})^{-1}:i<j<\omega\}\ra$ for the commutator $[(x_i,0),(x_j,0)]$ of $(x_i,0)$ and $(x_j,0)$. 

{\it The Mekler group of $C$}, denoted by $G_p(C)$, is defined as follows;
\[
G_p(C) := F_p(C)/V\text{ where } V:=\la \{(0,y_{i,j}):i<j<\omega, C\models R(c_i,c_j)\}\ra.
\]
We write $G(C):=G_p(C)$ if $p$ is clear from the context. Note that $G(C)$ is 2-nilpotent and of exponent $p$.
\end{definition}

\begin{definition}
Let $C$ be a countably infinite nice graph and $G:=G(C)$.
Let $g$ and $h$ be elements in $G$, $C(g)$ and $C(h)$ the centralizers of $g$ and $h$, respectively, and $Z:=Z(G)$ the center of $G$.
\begin{enumerate}
  \item[(i)] By $g \sim h$, we mean $C(g) = C(h)$.
  \item[(ii)] By $g \approx h$, we mean that $h = g^{r}c$ for some $c \in Z$ and some non-negative integer $r$.
  \item[(iii)] By $g \equiv_Z h$, we mean $gZ = hZ$.
\end{enumerate}
\end{definition}

\begin{definition}
For fixed nice graph $C$ and $G:=G(C)$, let $g$ be an element of $G$ and $Z:=Z(G)$ the center of $G$.
\begin{enumerate}
  \item[(i)] We say that $g$ is {\it isolated} if $g \approx h$ for any $h \in C(g) \cap (G \setminus Z)$. Otherwise, we say that $g$ is {\it non-isolated}.
  \item[(ii)] $n(g)$ is the number of $\approx$-classes in the $\sim$-class of $g$.
\end{enumerate}
\end{definition}

\begin{fact}\cite[Theorem A.3.10]{Hod93}\label{fact: 5 definable sets in G(C)}
For a nice graph $C$, $G:=G(C)$ can be partitioned into the following five $0$-definable sets:
\begin{enumerate}
\item[(i)] $Z(G)$,
\item[(ii)] $1^\nu$, the set of non-isolated elements $g \in G \setminus Z(G)$ such that $n(g)=1$,
\item[(iii)] $1^\iota$, the set of isolated elements $g \in G \setminus Z(G)$ such that $n(g)=1$,
\item[(iv)] $P$, the set of elements $g \in G$ such that $n(g)=p$,
\item[(v)] the set of elements $g \in G$ such that $n(g)=p-1$.
\end{enumerate}
\end{fact}

\begin{definition}
Continuing the notation of \Cref{fact: 5 definable sets in G(C)}, let $g$ be an element of $P$. For $h \in 1^\nu$, we say that $h$ is a {\it handle} of $g$ if $gh = hg$.
\end{definition}

\begin{fact}\cite[Theorem A.3.10 (c)]{Hod93}\label{fact: handle is unique}
Suppose $h$ is a handle of $g$ and $h'\in 1^\nu$. Then $h'\sim h$ if and only if $h'$ is also a handle of $g$.
\end{fact}

\begin{fact}\cite[the paragraph between Fact 2.5 and Definition 2.6]{CH18}
Let $C$ be a nice graph and $G:=G(C)$. Both $Z(G)$ and the quotient group $G/Z(G)$ are isomorphic to $\mathbb{F}_p$-vector spaces. Thus they are stable in the group language $\{+,-,0\}$.
\end{fact}

\begin{definition}
Let $B$ be a subgroup of $G$ containing $Z(G)$. A set $C \subseteq G$ is said to be {\it independent modulo $B$} if it is linearly independent over $B$ in the corresponding vector space.
\end{definition}

For any model $\mathbb{G}$ of $\Th(G(C))$, we denote by
$1_\mathbb{G}^\nu$, $1_\mathbb{G}^\iota$, and $P_\mathbb{G}$
the subsets of $\mathbb{G}$ defined by the formulas described in
\Cref{fact: 5 definable sets in G(C)} (ii), (iii), and (iv), respectively. When there is no risk of confusion, we simply write
$1^\nu$, $1^\iota$, and $P$.

\begin{definition}
Let $\mathbb{G}$ be a model of $\Th(G(C))$.
\begin{enumerate}
\item[(i)] A {\it $1^\nu$-transversal} of $\mathbb{G}$, denoted by $X^\nu$, is a set containing exactly one representative from each $\sim$-class of $1^\nu$.
\item[(ii)] An element of $\mathbb{G}$ is called {\it proper} if it is not a product of elements in $1^\nu$.
\item[(iii)] A {\it $p$-transversal} of $\mathbb{G}$, denoted by $X^p$, is a set consisting of proper elements in $P$ such that
\begin{enumerate}
\item[$\ast$] for any $g,h \in X^p$, $g \not\sim h$;
\item[$\ast$] $X^p$ is maximal with the property that for any finite subset $X' \subseteq X^p$, if all elements of $X'$ have the same handle, then $X'$ is independent modulo the subgroup $\langle Z(\mathbb{G}) \cup 1^\nu \rangle$.
\end{enumerate}
\item[(iv)] A {\it $1^\iota$-transversal} of $\mathbb{G}$, denoted by $X^\iota$, is a set containing exactly one representative in each $\sim$-class of $1^\iota$ and is maximal independent modulo the subgroup $\la Z(\mathbb{G})\cup 1^\nu \cup P\ra$.
\item[(v)] A transversal of $\mathbb{G}$ is a union of some $X^\nu$, $X^p$, and $X^\iota$ of $\mathbb{G}$.
\end{enumerate}
\end{definition}

\begin{fact}\label{fact: a nice graph is interpretable in its Mekler group}\cite[Theorem A.3.14 (a)]{Hod93}
Let $C$ be a nice graph and $G:=G(C)$ the Mekler group of $C$. Let $\Gamma:=\Gamma(G)$ be an graph interpretation in $G$ such that 
\begin{itemize}
\item[(i)] the universe is $\sim$-equivalent classes of the set of $1^\nu$,
\item[(ii)] the edge relation is the set of all pairs $([g]_\sim,[h]_\sim)$ such that $[g]_\sim \!\neq [h]_\sim$ and $gh=hg$.
\end{itemize}
Then $\Gamma\models\Th(C)$. 
\end{fact}

\begin{fact}\cite[Theorem A.3.14, Corollary A.3.15]{Hod93}\cite[Fact 2.11, Remark 2.12]{CH18}\label{fact: G=<X>H saturated}
Let $C$ be a nice graph and $\mathbb{G}$ a model of theory of $G(C)$. Then $\mathbb{G}$ is isomorphic to $\la X\ra \times H$ where $X$ is a transversal of $\mathbb{G}$ and $H$ is a subgroup of the center of $\mathbb{G}$. Moreover, $H$ is isomorphic to a $\mathbb{F}_p$-vector space. If $\mathbb{G}$ is uncountable and saturated additionally, then the following hold.
\begin{enumerate}
\item[(i)] Both $X^\nu$ and $H$ are saturated, as a graph structure and a group structure, respectively.
\item[(ii)] For any $g \in X^\nu$, the set $\{g' \in X^p : g \text{ is a handle of } g'\}$ is uncountable.
\item[(iii)] $|X^\iota|$ is uncountable.
\end{enumerate}
\end{fact}

We may regard $\mathbb{G}$ as isomorphic to $\langle X\rangle \times \langle H’ \rangle$, where $H’$ is a basis of $H$ as a vector space. Also note that $\Th(H)$ is stable and has quantifier elimination.

\begin{fact}\cite[Proposition 2.18]{CH18}\label{fact: X H are type definable}
Let $C$ be an infinite nice graph and $\mathbb{G}\models\Th(G(C))$.
There exists a partial type $\pi(\bar x,\bar y)$ with small tuples $\bar x,\bar y$
such that $\bar a\bar b \models \pi$ if and only if there exists a transversal $X$
of $G$ containing $\bar a$ and an independent subset $H$ of $Z(\mathbb{G})$ containing
$\bar b$ such that $\mathbb{G} = \langle X \rangle \times \langle H \rangle$. 
\end{fact}

\begin{fact}\cite[Lemma 2.14]{CH18}\label{fact: bijection->automorphism}
Let $C$ be an infinite nice graph and let $\mathbb{G}$, $X$, and $H$ be as in \Cref{fact: X H are type definable}.
Let $f : Y \to Z$ be a bijection between two small subsets $Y,Z \subseteq X$
such that:
\begin{enumerate}
\item[(i)] $f$ preserves transversal-types of $Y$ to $Z$;
\item[(ii)] $f$ preserves the handle relation, i.e. if $g \in X^\nu$ is a handle of $g' \in X^p$, then $f(g)$ is a handle of $f(g')$;
\item[(iii)] $Y^\nu$ and $Z^\nu$ have the same first-order type with respect to the graph structure on $X^\nu$, which is given by \Cref{fact: a nice graph is interpretable in its Mekler group}.
\end{enumerate}
Then $f$ extends to an automorphism of $\mathbb{G}$.
Moreover, for any $h,k \in H$, if $\tp(h)$ and $\tp(k)$ are equivalent modulo
$\Th(H)$, we may assume that the automorphism sends $h$ to $k$.

In particular, if two finite tuples $\bar{a}:=\bar{a}^\nu{}^\frown\bar{a}^p{}^\frown\bar{a}^\iota{}^\frown\bar{a}^h$, $\bar{b}=\bar{b}^\nu{}^\frown\bar{b}^p{}^\frown\bar{b}^\iota{}^\frown\bar{b}^h$ with $\bar{a}^\nu,\bar{b}^\nu\in X^\nu$, $\bar{a}^p,\bar{b}^p\in X^p$, $\bar{a}^\iota,\bar{b}^\iota\in X^\iota$, $\bar{a}^h,\bar{b}^h \in H$, $\bar{a}^\nu:=(a^\nu_0,...,a^\nu_{n-1})$, $\bar{a}^p:=(a^p_0,...,a^p_{m-1})$, $\bar{b}^\nu:=(b^\nu_0,...,b^\nu_{n-1})$, $\bar{b}^p:=(b^p_0,...,b^p_{m-1})$, and $(|\bar{a}^\iota|,|\bar{a}^h|)=(|\bar{b}^\iota|,|\bar{b}^h|)$ satisfy
\begin{itemize}
\item[(iv)] $\bar{a}^\nu$ and $\bar{b}^\nu$ have the same type modulo $\Th(X^\nu)$, with respect to the graph structure on $X^\nu$ given by \Cref{fact: a nice graph is interpretable in its Mekler group},
\item[(v)] $\bar{a}^h$ and $\bar{b}^h$ have the same type modulo $\Th(H)$, with respect to the group structure on $H$,
\item[(vi)] $a^\nu_i$ is a handle of $a^p_j$ if and only if $b^\nu_i$ is a handle of $b^p_j$ for each $i<n$ and $j<m$,
\end{itemize}
then $\bar{a}\equiv\bar{b}$.
\end{fact}

Every nice graph is interpretable in its Mekler group as mentioned in \Cref{fact: a nice graph is interpretable in its Mekler group},
but the converse does not hold in general: a nice graph and its Mekler group are not bi-interpretable. Moreover, some model-theoretic properties may fail to hold in a Mekler group, even when the original nice graph satisfies them. For example, Mekler groups are never distal, even if the underlying nice graph is distal \cite{BPT24}.
 
\smallskip  Nevertheless, a growing list of model-theoretic properties is known to be preserved under this uniform way of constructing groups. We leave some of them below.

\begin{fact}
Let $C$ be a nice graph and $G(C)$ its Mekler group. 
\begin{itemize}
\item[(i)] \cite{Mek81} $\Th(C)$ is stable if and only if $\Th(G(C))$ is stable.
\item[(ii)] \cite{Bau02} $\Th(C)$ is $CM$-trivial if and only if $\Th(G(C))$ is $CM$-trivial.
\item[(iii)] \cite{Bau02} $\Th(C)$ is simple if and only if $\Th(G(C))$ is simple.
\item[(iv)] \cite{CH18} $\Th(C)$ is NIP if and only if $\Th(G(C))$ is NIP.
\item[(v)] \cite{CH18} $\Th(C)$ is NIP$_n$ if and only if $\Th(G(C))$ is NIP$_n$ for each $n<\omega$.
\item[(vi)] \cite{CH18} $\Th(C)$ is NTP$_2$ if and only if $\Th(G(C))$ is NTP$_2$.
\item[(vii)] \cite{Ahn19} $\Th(C)$ is NTP$_1$ if and only if $\Th(G(C))$ is NTP$_1$.
\item[(viii)] \cite{AKL21} $\Th(C)$ is NATP if and only if $\Th(G(C))$ is NATP.
\end{itemize}
\end{fact}

More recently, a widely applicable preservation result has been established by Boissonneau, Papadopoulos, and Touchard.

\begin{fact}\cite{BPT24}
$\Th(C)$ collapses $\CI$-indiscernible if and only if  $\Th(G(C))$ collapses $\CI$-indiscernible. Thus any dividing line characterized via collapsing indiscernibles is preserved by Mekler’s construction. In particular,
\begin{itemize}
\item[(ix)] $\Th(C)$ is NFOP$_2$ if and only if $\Th(G(C))$ is NFOP$_2$.
\end{itemize}
\end{fact}

\section{Mekler's construction of NCTP theories}

In this section, we prove that Mekler’s construction preserves NCTP. As mentioned in the introduction, the key idea is to use the characterization of NCTP via tree indiscernibility, stated below.

\begin{theorem}\label{thm: criterion NCTP}
The following are equivalent.
\begin{itemize}
\item[(i)] $T$ is NCTP.
\item[(ii)] For any parameter $b$, a cardinal $\kappa>2^{|T|+|b|+\aleph_0}$ with $\cf(\kappa)=\kappa$, and a strong indiscernible tree $(a_\eta)_{\eta\in\omega^{<\kappa}}$, there exist $b'$ and $i<\kappa$  such that 
\begin{itemize}
\item[$\ast$] $b'\equiv_{a_{0^i{}^\frown(1)}}\!b$,
\item[$\ast$] $(a_{0^i{}^\frown(1)^\frown 0^j})_{j<\kappa}$ is order indiscernible over $b'$.
\end{itemize}
\end{itemize}
\begin{proof}
Suppose (i). Choose any $b$, $\kappa>2^{|T|+|b|+\aleph_0}$ with $\cf(\kappa)=\kappa$, and a strongly indiscernible tree $(a_\eta)_{\eta\in \omega^{<\kappa}}$. Since $\kappa$ is sufficiently large, there exists $I\subseteq \kappa$ with $|I|=\kappa$ such that $a_{0^i{}^\frown(1)}\equiv_ba_{0^{i'}{}^\frown(1)}$ for all $i,i'\in I$. Choose any $i\in I$ and let $p(x,y):=\tp(b,a_{0^i{}^\frown(1)})$. If $\cup_{j<\kappa}p(x,a_{0^i{}^\frown(1)\coc0^j})$ is inconsistent, then there exists a witness of CTP by strong indiscernibility of $(a_\eta)_{\eta\in\omega^{<\kappa}}$. Thus $\cup_{j<\kappa}p(x,a_{0^i{}^\frown(1)\coc0^j})$ is consistent. By Ramsey, we can find $b'\models \cup_{j<\kappa}p(x,a_{0^i{}^\frown(1)\coc0^j})$ such that $(a_{0^i{}^\frown(1)\coc0^j})_{j<\kappa}$ is $b'$-indiscernible.

Now we suppose (ii). If $T$ has CTP, then by compactness, there exist $\varphi(x,y)$ and a strong indiscernible tree $(a_\eta)_{\eta\in \omega^{<\kappa}}$ satisfying the conditions in \Cref{def: ATP CTP BTP} (ii), for some $\kappa>2^{|T|+\aleph_0}$. Then there exists $b\models \{\varphi(x,a_\eta)\}_{\eta\in D}$. Note that $b$ is a finite tuple and hence $\kappa>2^{|T|+|b|+\aleph_0}$. By (ii), there exist $i<\kappa$ and  $b'\equiv_{a_{0^i{}^\frown(1)}}\!b$ such that $(a_{0^i{}^\frown(1)\coc0^j})_{j<\kappa}$ is $b'$-indiscernible. Thus $\{\varphi(x,a_{0^i{}^\frown(1)\coc0^j})\}_{j<\kappa}$ is consistent, so we get a contradiction.
\end{proof}
\end{theorem}

We will also need the following lemma concerning colorings on trees whose height is sufficiently large.

\begin{lemma}\label{lem: coloring on omega-kappa}
Let $\kappa$, $\lambda$, and $\mu$ be cardinals with $\mu<\cf(\kappa)$. For any map $c$ from $\lambda^{<\kappa}$ to $\mu$, there exist $\eta\in \lambda^{<\kappa}$ and $\theta<\mu$ such that for all $\nu\unrhd\eta$, there exists $\rho\unrhd\nu$ such that \[
\{\xi\in\lambda^{<\kappa}:\xi\unrhd\rho^\frown(i),c(\xi)=\theta\}\neq\emptyset
\]
for each $i<\lambda$.
\begin{proof}
By using the same argument in \cite[Remark 3.22]{AKL21}.
\end{proof}
\end{lemma}

\begin{theorem}\label{thm: Mekler group NCTP}
Let $C$ be a nice graph and $G(C)$ its Mekler group. Then $\Th(C)$ is NCTP if and only if $\Th(G(C))$ is NCTP.
\begin{proof}
The implication from the right to the left is clear since $C$ is interpretable in $G(C)$.

Suppose $\Th(C)$ is NCTP and $\Th(G(C))$ has CTP. Fix a sufficiently saturated model $\mathbb{G}$ of $\Th(G(C))$, a transversal $X:=X^\nu{}\cup X^p\cup X^\iota$ of $\mathbb{G}$, and a subgroup $H$ of $Z(\mathbb{G})$ with $\mathbb{G}=\la X\ra\times H$. Let $\kappa$ be a cardinal such that $\kappa>2^{|T|+|\aleph_0|}$ and $\cf(\kappa)=\kappa$. Assume that CTP is witnessed by a formula $\varphi(\bar{x},\bar{y})$ and $(\bar{a}_\eta)_{\eta\in\omega^{<\kappa}}$. For each $\eta\in\omega^{<\kappa}$, there exists a tuple of terms $\bar{t}_\eta$ in the group language such that $\bar{a}_\eta=\bar{t}_\eta(\bar{a}'_\eta)$ for some $\bar{a}'_\eta\in X\cup H$. By \Cref{lem: coloring on omega-kappa}, we can give a $\{\unlhd,\lex\}$-embedding $f$ from $\omega^{<\kappa}$ to $\omega^{<\kappa}$ such that $f(\eta)\wedge f(\eta')\unrhd f(\eta\wedge\eta')$ and $\bar{t}_{f(\eta)}=\bar{t}_{f(\eta')}$ for all $\eta,\eta'\in\omega^{<\kappa}$. Then the image of any given descending comb is a descending comb, the image of any path is still a path. Thus we may assume $\bar{a}_\eta\in X\cup H$ for all $\eta\in\omega^{<\kappa}$, by replacing $\varphi$ with a new formula obtained by adding the term $\bar{t}_{f(\eta)}$ to $\varphi$. Moreover, by applying \Cref{lem: coloring on omega-kappa} again, there exist $(n^\nu, n^p, n^\iota, n^h)\in\omega^4$ such that for each $\eta\in\omega^{<\kappa}$, $\bar{a}_\eta$ is of the form $\bar{a}^\nu_\eta{}^\frown\bar{a}^p_\eta{}^\frown\bar{a}^\iota_\eta{}^\frown\bar{a}^h_\eta$, where $\bar{a}^\nu_\eta:=(a^\nu_{\eta,0},...,a^\nu_{\eta,n^\nu-1})\in X^\nu$, $\bar{a}^p_\eta:=(a^p_{\eta,0},...,a^p_{\eta,n^p-1})\in X^p$, $\bar{a}^\iota_\eta:=(a^\iota_{\eta,0},...,a^\iota_{\eta,n^\iota-1})\in X^\iota$, and $\bar{a}^h_\eta:=(a^h_{\eta,0},...,a^h_{\eta,n^h-1})\in H$. We may assume that $a^\nu_{\eta,i}\neq a^\nu_{\eta,j}$ for each $\eta\in\omega^{<\kappa}$ and $i\neq j<n^\nu$. By applying modeling property, we may assume $(a_\eta)_{\eta\in\omega^{<\kappa}}$ is strongly indiscernible.
By \Cref{fact: handle is unique} and \Cref{fact: X H are type definable}, we may assume that 
\begin{itemize}
\item[(i)] $\{h\in X^\nu: h\text{ is a handle of }a^p_{\eta,i}\text{ for some }i<n^p\}\subseteq \bar{a}^\nu_\eta$ for each $\eta\in\omega^{<\kappa}$.
\end{itemize}

Let $\bar{b}\models\{\varphi(\bar{x},\bar{a}_{0^i{}^\frown(1)})\}_{i<\kappa}$. Then there exist a term $\bar{s}$ and $\bar{c}:=\bar{c}^\nu{}^\frown\bar{c}^p{}^\frown\bar{c}^\iota{}^\frown\bar{c}^h$ with $\bar{c}^\nu\in X^\nu$, $\bar{c}^p\in X^p$, $\bar{c}^\iota\in X^\iota$, and $\bar{c}^h\in H$ such that $\bar{b}=\bar{t}(\bar{c})$. Note that $\bar{c}\models\{\varphi(\bar{s}(\bar{z}),a_{0^i{}^\frown(1)})\}_{i<\kappa}$ and $\varphi(\bar{s}(\bar{z}),\bar{y})$ also witnesses CTP with $(\bar{a}_\eta)_{\eta\in\omega^{<\kappa}}$. Let $\bar{c}^p:=(c^p_0,...,c^p_{m-1})$. If $d$ is a handle of $c^p_i$ for some $i<m$, then by \Cref{fact: handle is unique} again, we may assume that $d$ satisfies one of the following.
\begin{itemize}
\item[(ii)] there exists $j<n^\nu$ such that $d=a^\nu_{\eta,j}$ for all $\eta\in\omega^{<\kappa}$,
\item[(iii)] $d\neq a^\nu_{\eta,j}$ for all $\eta\in\omega^{<\kappa}$ and $j<n^\nu$.
\end{itemize}
Since $\kappa$ is sufficiently large, we may assume that
\begin{itemize}
\item[(iv)] $(\bar{a}^h_{0^{i}{}^\frown(1)^\frown 0^j})_{j<\omega}\equiv_{\bar{c}^h}(\bar{a}^h_{0^{i'}{}^\frown(1)^\frown 0^j})_{j<\omega}$ modulo $\Th(H)$ for all $i,i'<\kappa$, where $\Th(H)$ is a theory of the group structure on $H$.
\end{itemize}

Note that $X^\nu$ interprets a graph structure whose theory is $\Th(C)$, which is NCTP, and $(\bar{a}_\eta^\nu)_{\eta\in\omega^{<\kappa}}$ is still strongly indiscernible modulo $\Th(X^\nu)$. Thus by \Cref{thm: criterion NCTP} and \Cref{fact: G=<X>H saturated}, there exist $i^*<\kappa$ and $\bar{c}^{\nu*}\in X^\nu$ such that 
\begin{itemize}
\item[(v)] $\bar{c}^{\nu*}\bar{a}^\nu_{0^{i*}{}^\frown(1)}$ and $\bar{c}^\nu\bar{a}^\nu_{0^{i*}{}^\frown(1)}$ have the same type modulo $\Th(X^\nu)$,
\item[(vi)] $(\bar{a}^\nu_{0^{i*}{}^\frown(1)^\frown 0^j})_{j<\omega}$ is order indiscernible over $\bar{c}^{\nu*}$ modulo $\Th(X^\nu)$.
\end{itemize}

On the other hand, $\Th(H)$ is also NCTP since it is stable, and $(\bar{a}_\eta^h)_{\eta\in\omega^{<\kappa}}$ is strongly indiscernible modulo $\Th(H)$. Thus by \Cref{thm: criterion NCTP} and \Cref{fact: G=<X>H saturated} again, there exist $i^\dagger<\kappa$ and $\bar{c}^{h\dagger}\in H$ such that
\begin{itemize}
\item[(vii)] $\bar{c}^{h\dagger}\bar{a}^h_{0^{i\dagger}{}^\frown(1)}$ and $\bar{c}^h\bar{a}^h_{0^{i\dagger}{}^\frown(1)}$ have the same type modulo $\Th(H)$,
\item[(viii)] $(\bar{a}^h_{0^{i\dagger}{}^\frown(1)^\frown 0^j})_{j<\omega}$ is order indiscernible over $\bar{c}^{h\dagger}$ modulo $\Th(H)$.
\end{itemize}
By (iv), we may assume $i^*=i^\dagger$. Put $\bar{c}^{h*}:=\bar{c}^{h\dagger}$. By (i) and \Cref{fact: bijection->automorphism}, there is an automorphism $\sigma$ on $\mathbb{G}$ sending $\bar{c}^\nu\bar{c}^\iota\bar{c}^h$ to $\bar{c}^{\nu*}\bar{c}^\iota\bar{c}^{h*}$ over $\bar{a}_{0^{i*}{}^\frown(1)}$. Let $\bar{c}^{p*}:=\sigma(\bar{c}^p)$ and $\bar{c}^*:=\bar{c}^{\nu*}{}^\frown\bar{c}^{p*}{}^\frown\bar{c}^\iota{}^\frown\bar{c}^{h*}$. Then $\bar{c}^{*}\models\varphi(\bar{s}(\bar{z}),\bar{a}_{0^{i*}{}^\frown(1)})$. By (ii), (iii), and \Cref{fact: bijection->automorphism} again, we have $\bar{a}_{0^{i*}{}^\frown(1)}\bar{c}^*\equiv\bar{a}_{0^{i*}{}^\frown(1)^\frown0^j}\bar{c}^*$ for each $j<\kappa$. Thus $\bar{c}^*\models\{\varphi(\bar{s}(\bar{z}),\bar{a}_{0^{i*}{}^\frown(1)^\frown 0^j})\}_{j<\kappa}$. This is a contradiction.
\end{proof}
\end{theorem}

\section{Mekler's construction of NBTP theories}

In this section, we show that Mekler’s construction preserves NBTP. The overall strategy is similar to the NCTP case. For technical convenience, we first slightly refine the definition of NBTP. More precisely, we simplify the consistency part appearing in the original definition and divide the inconsistency part into two separate patterns. This reformulation makes it more suitable to apply the appropriate notion of indiscernibility for NBTP.

\begin{definition}\label{def: strict left/right-leaning/veering, direct sibling}
Let $X\subseteq \omega^{<\omega}$ and $\kappa$ a cardinal.
\begin{itemize}
\item[(i)] $X$ is a {\it strict left-leaning path} if $\emptyset\notin X$ and there is some enumeration $\{\eta_i\}_{i<\kappa}$ of $X$ such that $\eta_i^- {^\frown} (j)\lhd\eta_{i+1}$ for some $j< t(\eta_i)$ for all $i<\kappa$. 
\item[(ii)] $X$ is a {\it strict right-veering path} if $\emptyset\notin X$ and there is some enumeration $\{\eta_i\}_{i<\kappa}$ of $X$ such that $\eta_i^- {^\frown} (j)\lhd \eta_{i+1}$ for some $j> t(\eta_i)$ for all $i<\kappa$.
\item[(iii)] $X$ is a {\it direct sibling set} if $\emptyset\notin X$ and there is some enumeration $\{\eta_i\}_{i<\kappa}$ of $X$ such that $\eta_i^- {^\frown} (j)= \eta_{i+1}$ for some $j>t(\eta_i)$ for all $i<\kappa$.
\end{itemize}
\end{definition}

\begin{lemma}\label{lem: weak BTP}
$T$ has BTP if and only if there exist $\varphi(x,y)$, $(a_\eta)_{\eta\in\omega^{<\omega}}$, and $k<\omega$ such that
\begin{itemize}
\item[(i)] $\{\varphi(x,a_\eta)\}_{\eta\in X}$ is consistent for any strict left-leaning path $X$,
\item[(ii)] $\{\varphi(x,a_\eta)\}_{\eta\in X}$ is $k$-inconsistent for any strict right-veering path $X$,
\item[(iii)] $\{\varphi(x,a_\eta)\}_{\eta\in X}$ is $k$-inconsistent for any direct sibling set $X$.
\end{itemize}
\begin{proof}
The implication from the left to the right is trivial. Assume that $\varphi(x,y)$, $(a_\eta)_{\eta\in\omega^{<\omega}}$, and $k<\omega$ satisfy (i) and (ii). We define a map $f$ from $\omega^{<\omega}$ to $\omega^{<\omega}$ by
\[
    f(\eta)= 
\begin{cases}
    (1) & \text{if } \;\; l(\eta)=0\\
    (0)^\frown(2 t(\eta)+1) & \text{if } \;\; l(\eta)=1\\
    f(\eta^-)^-{}^\frown(2 t(\eta^-))^\frown (2 t(\eta)+1)  & \text{if } \;\; l(\eta)>1
    \end{cases}
\]
For each $\eta\in\omega^{<\omega}$, let $b_\eta:=a_{f(\eta)}$. It is easy to check that $\{\varphi(x,b_\eta)\}_{\eta\in X}$ is consistent for any left-leaning path $X$, and that $\{\varphi(x,b_\eta)\}_{\eta\in X}$ is $k^2$-inconsistent for any right-veering path $X$.
\end{proof}
\end{lemma}

\begin{definition}
We say $\varphi(x,y)$ witnesses {\it BTP$^-$} with $(a_\eta)_{\eta\in\omega^{<\omega}}$ if they satisfy (i), (ii), and (iii) in \Cref{lem: weak BTP}.
\end{definition}

We now work with a notion of indiscernibility suitable for NBTP. The notion of s-indiscernibility is well known and has the modeling property as noted in \Cref{fact: s-indisc}. Moreover, applying s-modeling property to a witness of BTP yields an indiscernible tree without destroying BTP. The difficulty, however, is that in such an indiscernible tree, we cannot guarantee that the paths are order-indiscernible.

To overcome this issue, we use the technique developed in \cite{GIL02} and \cite{KR20}. This yields a notion of indiscernibility that is slightly stronger than s-indiscernibility.

\begin{notation}
Let $\kappa$ be a cardinal and $X\subset \kappa$. Then $\omega^{<\kappa}|_X$ is the set of all $\eta\in\omega^{<\kappa}$ such that $l(\eta)\in X$ and $\eta(i)=0$ for each $i\in \text{dom}(\eta)\setminus X$ where $\text{dom}(\eta)$ is the domain of $\eta$. Note that $\omega^{<\kappa}|_\emptyset=\emptyset$.
\end{notation}

\begin{definition}\label{def: level-s-indisc}
Let $\kappa$ be an infinite cardinal. We will say $(a_\eta)_{\eta\in\omega^{<\kappa}}$ is {\it level-s-indiscernible} if it is s-indiscernible and $(a_\eta)_{\eta\in\omega^{<\kappa}|_X}\equiv(a_\eta)_{\eta\in\omega^{<\kappa}|_Y}$ for all $X,Y\subseteq \kappa$ with $|X|=|Y|<\omega$.
\end{definition}

\begin{notation}\cite{KR20}
Let $\alpha$ be an ordinal and $\neg\text{lim}(\alpha)$ the set of all non-limit ordinals $\beta<\alpha$.
\begin{itemize}
\item[(i)] For each $\beta<\alpha$, $\CT_{\beta,\alpha}$ is the set of all functions $\eta$ from $[\beta,\alpha)$ to $\omega$ such that $|\{i\in \text{dom}(\eta):\eta(i)\neq0\}|<\omega$ where $\text{dom}(\eta)$ is the domain of $\eta$.
\item[(ii)] $\CT_\alpha:=\bigcup_{\beta\in\neg\text{lim}(\alpha)}\CT_{\beta,\alpha}$.
\end{itemize}
\smallskip
Let $\eta,\nu\in\CT_\alpha$, $\beta\in\neg\text{lim}(\alpha)$, and $X\subseteq\neg\text{lim}(\alpha)$. 
\begin{itemize}
\item[(iii)] By $\eta\unlhd\nu$, we mean $\eta \subseteq \nu$.
\item[(iv)] By $\eta\perp\nu$, we mean that $\eta\not\!\!\unlhd\,\nu$ and $\nu\not\!\!\unlhd\,\eta$.
\item[(v)] By $\eta\wedge\nu$, we mean the $\unlhd$-maximal $\xi\in\CT_\alpha$ such that $\xi\unlhd\eta$ and $\xi\unlhd\nu$.
\item[(vi)] By $l(\eta)$, we mean the least element of $\text{dom}(\eta)$.
\item[(vii)] $P_\beta:=\{\mu\in\CT_\alpha: l(\mu)=\beta\}$.
\item[(viii)] By $\CT_\alpha|_X$, we mean $\{\mu\in\CT_\alpha: \mu(i)=0\text{ for all }i\in\text{dom}(\mu)\setminus X, l(\mu)\in X\}$. 
\end{itemize}
Thus we can give an $\CL_s$-structure on $\CT_\alpha$ and consider level-s-indiscernibility of $\CT_\alpha$-indexed sets as in \Cref{def: level-s-indisc}. Note that $\age_{\CL_s}(\omega^{<\omega})=\age_{\CL_s}(\CT_\alpha)$ for each infinite ordinal $\alpha$.
\end{notation}

\begin{lemma}\label{lem: level-s-indiscernible witness of BTP}
If $T$ has BTP, then there exist $\varphi(x,y)$, level-s-indiscernible  $(a_\eta)_{\eta\in\omega^{<\omega}}\!$, and $k<\omega$ such that
\begin{itemize}
\item[(i)] $\{\varphi(x,a_\eta)\}_{\eta\in X}$ is consistent for any strict left-leaning path $X$ in $\omega^{<\omega}$,
\item[(ii)] $\{\varphi(x,a_\eta)\}_{\eta\in X}$ is $k$-inconsistent for any strict right-veering path $X$ in $\omega^{<\omega}$,
\item[(iii)] $\{\varphi(x,a_\eta)\}_{\eta\in X}$ is $k$-inconsistent for any direct sibling set $X$ in $\omega^{<\omega}$.
\end{itemize}
\begin{proof}
Suppose $T$ is BTP. Let $\lambda:=2^{\aleph_0+|T|}$ and $\kappa:=\beth_{\lambda^+}(\lambda)$. By using compactness, \Cref{fact: age(I)=age(J)}, and \Cref{lem: weak BTP}, we can find $\varphi(x,y)$, s-indiscernible $(b_\eta)_{\eta\in\CT_\kappa}$, and $k<\omega$ such that
\begin{itemize}
\item[(iv)] $\{\varphi(x,b_\eta)\}_{\eta\in X}$ is consistent for any strict left-leaning $X$ in $\CT_\kappa$,
\item[(v)] $\{\varphi(x,b_\eta)\}_{\eta\in X}$ is $k$-inconsistent for any strict right-veering $X$ in $\CT_\kappa$.
\item[(vi)] $\{\varphi(x,b_\eta)\}_{\eta\in X}$ is $k$-inconsistent for any direct sibling set $X$ in $\CT_\kappa$.
\end{itemize}

\medskip

\noindent{\it Claim.} We can find level-s-indiscernible $(c_\eta)_{\eta\in\CT_\omega}$ such that
\begin{itemize}
\item[(vii)] For all $X\subseteq \omega$ with $|X|<\omega$, there exists $Y\subseteq \kappa$ such that $|Y|=|X|$ and $(c_\eta)_{\eta\in\CT_\omega|_Y}\equiv(b_\eta)_{\eta\in\CT_\kappa|_X}$.
\end{itemize}

\medskip

\noindent{\it Proof of Claim.} We use the argument that appears in \cite[Lemma 5.10]{KR20}. The proof is essentially the same as theirs, but we include it below for completeness. For each $0<n<\omega$, let
\[
\Gamma_n:=\{\tp((b_\eta)_{\eta\in\Tau_\kappa|_X}):X\subseteq \neg\text{lim}(\kappa), |X|=n\}.
\]
By induction on $n<\omega$, we will find $(p_n,F_n,(X_{\xi,n})_{\xi\in F_n})_{n<\omega}$ such that

\begin{itemize}
\item[(viii)] $p_0=\emptyset$ and $p_n\in \Gamma_n$ for $n>0$,
\item[(ix)] \[
\Delta_n((x_\eta)_{\eta\in\Tau_\omega}):=\bigcup_{n'\le n}\bigcup_{\substack{X\subseteq\omega\\ |X|=n'}}p_{n'}((x_\eta)_{\eta\in\Tau_\omega|_X})
\]
is consistent for each $n<\omega$,
\item[(x)] $F_n$ is a cofinal subset of $\lambda^+$ for each $n<\omega$,
\item[(xi)] $X_{\xi,n}\subseteq\neg\text{lim}(\kappa)$ for each $n<\omega$ and $\xi\in F_n$, 
\item[(xii)] $F_{n+1}\subseteq F_n$ for each $n<\omega$,
\item[(xiii)] $|X_{\xi,n}|>\beth_\alpha(\lambda)$ when $\xi$ is the $\alpha$th element of $F_n$,
\item[(xiv)] $X_{\xi,n+1}\subseteq X_{\xi,n}$ for all $n<\omega$ and $\xi\in F_{n+1}$,
\item[(xv)] If $W\subseteq X_{\xi,n}$ and $|W|=n>0$, then $(b_\eta)_{\eta\in\Tau_\kappa|_W}\models p_n$,
\item[(xvi)] $|F_n|=\lambda^+$ for each $n<\omega$.
\end{itemize}
For $n=0$, we let $p_0:=\emptyset$, $F_0:=\lambda^+$, and $X_{\xi,0}:=\neg\text{lim}(\kappa)$ for each $\xi\in F_0$. Suppose $p_n$, $F_n$, and $(X_{\xi,n})_{\xi\in F_n}$ have been constructed. Write $F_n:=\{\xi_\alpha:\alpha<\lambda^+\}$ where $\xi_\alpha$ is the $\alpha$th element of $F_n$. Then for all $\alpha<\lambda^+$, we have $|X_{\xi_{\alpha+n+1},n}|>\beth_{\alpha+n+1}(\lambda)$. For each $\alpha<\lambda^+$, define a $\lambda$-coloring on $\{W\subseteq X_{\xi_{\alpha+n+1},n}:|W|^{n+1}\}$ by $W\mapsto\tp( (b_\eta)_{\eta\in\Tau_\kappa|_W})$. By Erd\"{o}s-Rado, there is a homogeneous subset $X_{\xi_{\alpha+n+1},n+1}\subseteq X_{\xi_{\alpha+n+1},n}$ with $|X_{\xi_{\alpha+n+1},n+1}|>\beth_\alpha(\lambda)$. Let $p_{n+1,\alpha+n+1}$ denote its constant value. By pigeon hole principle, there is $Y\subseteq\{\alpha+n+1:\alpha<\lambda^+\}$ of cardinality $\lambda^+$ so that $\beta,\beta'\in Y$ implies $p_{n+1,\beta}=p_{n+1,\beta'}$. Let $p_{n+1}=p_{n+1,\beta}$ for some/all $\beta\in Y$. Put $F_{n+1}=\{\xi_\beta:\beta\in Y\}$. Then $p_{n+1}$, $F_{n+1}$, and $(X_{\xi,n})_{\xi\in F_{n+1}}$ satisfy the requirements. 
By compactness, $\Delta((x_\eta)_{\eta\in\Tau_\omega})$ is consistent. Let $(c_\eta)_{\eta\in\Tau_\omega}$ be its realization. Then $(c_\eta)_{\eta\in\Tau_\omega}$ is level-s-indiscernible and satisfies (vii). 
$\dashv$

\medskip

Note that $\varphi$ still witnesses BTP with $(c_\eta)_{\eta\in\Tau_\omega}$ by (vii).
Using $(c_\eta)_{\eta\in\CT_\omega}$ and compactness, we can find level-s-indiscernible $(a_\eta)_{\eta\in\omega^{<\omega}}$ satisfying (i), (ii), and (iii).
\end{proof}
\end{lemma}

\begin{theorem}\label{thm: criterion NBTP}
The following are equivalent.
\begin{itemize}
\item[(i)] $T$ is NBTP.
\item[(ii)] For any parameter $b$, a cardinal $\kappa>2^{|T|+|b|+|\aleph_0|}$ with $\cf(\kappa)=\kappa$, and a level-s-indiscernible tree $(a_\eta)_{\eta\in\omega^{<\kappa}}$, there exist $i<\kappa$ and $b'$ such that 
\begin{itemize}
\item[$\ast$] $b'\equiv_{a_{0^i{}^\frown(1)}}\!b$,
\item[$\ast$] $(a_{0^i{}^\frown(j+1)})_{j<\omega}$ or $(a_{0^i{}^\frown(2)^j{}^\frown(1)})_{j<\kappa}$ is order indiscernible over $b'$.
\end{itemize}
\end{itemize}
\begin{proof}
Suppose $T$ is NBTP and choose any $b$, a cardinal $\kappa>2^{|T|+|b|+|\aleph_0|}$ with $\cf(\kappa)=\kappa$, and a level-s-indiscernible tree $(a_\eta)_{\eta\in\omega^{<\kappa}}$. Since $\kappa$ is sufficiently large, there exists $X\subseteq \kappa$ with $|X|=\kappa$ such that $a_{0^i{}^\frown(1)}\equiv_b a_{0^j{}^\frown(1)}$ for all $i,j\in X$. Choose any $i\in X$ and let $p(x,y):=\tp(b, a_{0^i{}^\frown(1)})$. We will show that $\bigcup_{j<\omega}p(x,a_{0^i{}^\frown(j+1)})$ or $\bigcup_{j<\omega}p(x,a_{0^i{}^\frown(2)^j{}^\frown(1)})$ is consistent. Suppose not. Then by level-s-indiscernibility, there exist $\psi_0(x,y),\psi_1(x,y)\in p(x,y)$ and $k_0,k_1<\omega$ such that $\{\psi_0(x,a_{0^i{}^\frown(j+1)})\}_{j<\omega}$ is $k_0$-inconsistent and $\{\psi_1(x,a_{0^i{}^\frown(2)^j{}^\frown(1)})\}_{j<\omega}$ is $k_1$-inconsistent. Let $\psi(x,y):=\psi_0(x,y)\wedge\psi_1(x,y)$. It is easy to check that the set $\{\psi(x,a_\eta)\}_{\eta\in Y}$ is $\max\{k_0,k_1\}$-inconsistent for any strict right-veering path $Y$ and direct sibling set $Y$, starting with $0^i{}^\frown(1)$. On the other hand,  $\psi(x,y)\in p(x,y)$ and $b\models\{\psi(x,a_{0^i{}^\frown(1)})\}_{i\in X}$. By level-s-indiscernibility and \Cref{lem: level-s-indiscernible witness of BTP}, $\psi(x,y)$ witnesses BTP with $(a_\eta)_{\eta\in\omega^{<\kappa}}$, a contradiction. Thus $\bigcup_{j<\omega}p(x,a_{0^i{}^\frown(j+1)})$ or $\bigcup_{j<\omega}p(x,a_{0^i{}^\frown(2)^j{}^\frown(1)})$ is consistent. By standard argument using Ramsey theorem, we can find $b'$ satisfying the conditions we want.

Now we show the converse. Suppose $T$ has BTP. Let $\kappa$ be a cardinal larger than $2^{|T|+|\aleph_0|}$ with $\cf(\kappa)=\kappa$. Then by \Cref{lem: level-s-indiscernible witness of BTP} and compactness, there exist $\varphi(x,y)$ and level-s-indiscernible $(a_\eta)_{\eta\in\omega^{<\kappa}}$ that witness BTP$^-$. Choose any $b\models\{\varphi(x,a_{0^i{}^\frown(1)})\}$. Then there do not exist $i<\kappa$ and $b'\equiv_{0^i{}^\frown(1)}b$ such that $(a_{0^i{}^\frown(j+1)})_{j<\omega}$ or $(a_{0^i{}^\frown(2)^j{}^\frown(1)})_{j<\kappa}$ is order indiscernible over $b'$.
\end{proof}
\end{theorem}

\begin{lemma}\label{lem: simple level-s-indiscernible}
Suppose $T$ is simple. Then for any cardinal $\kappa>2^{|T|+|b|+|\aleph_0|}$ with $\cf(\kappa)=\kappa$, a level-s-indiscernible tree $(a_\eta)_{\eta\in\omega^{<\kappa}}$, and a parameter $b$, there exist   $b'$, $b''$, and $i<\kappa$ such that
\begin{itemize}
\item[(i)] $b\equiv_{a_{0^i{}^\frown(1)}}\!b'\equiv_{a_{0^i{}^\frown(1)}}\!b''$,
\item[(ii)] $(a_{0^i{}^\frown(1+j)})_{j<\omega}$ is order indiscernible over $b'$,
\item[(iii)] $(a_{0^i{}^\frown(2)^j{}^\frown(1)})_{j<\kappa}$ is order indiscernible over $b''$.
\end{itemize} 
\begin{proof}
Suppose $T$ is simple and choose $\kappa$, $(a_\eta)_{\eta\in\omega^{<\kappa}}$, and $b$ as above. Then there exists $X\subseteq \kappa$, with $|X|=\kappa$ such that $a_{0^i{}^\frown(1)}\equiv_b a_{0^j{}^\frown(1)}$ for all $i,j\in X$. Choose any $i\in X$ and let $p(x,y):=\tp(b,a_{0^i{}\frown(1)})$. It is enough to show that $\bigcup_{j<\omega}p(x,a_{0^i{}^\frown(j+1)})$ and $\bigcup_{j<\omega}p(x,a_{0^i{}^\frown(2)^j{}^\frown(1)})$ is consistent. To get a contradiction, first we assume that $\bigcup_{j<\omega}p(x,a_{0^i{}^\frown(j+1)})$ is not consistent. Then there exists $\psi_0(x,y)\in p(x,y)$ such that
\begin{itemize}
\item[(iv)] $\{\psi_0(x,a_\eta)\}_{\eta\in X}$ is consistent for any strict left-leaning $X$ in $\omega^{<\kappa}$,
\item[(v)] $\{\psi_0(x,a_\eta)\}_{\eta\in X}$ is inconsistent for any direct sibling $X$ in $\omega^{<\kappa}$.
\end{itemize} 
Recall the map $f:\omega^{<\omega}\to\omega^{<\omega}$ in \Cref{lem: weak BTP} and put $b_\eta:=a_{f(\eta)}$ for each $\eta\in\omega^{<\omega}$. Then $\psi_0(x,y)$ witnesses tree property with $(b_\eta)_{\eta\in\omega^{<\omega}}$. Since we assume $T$ is simple, it is a contradiction. $\bigcup_{j<\omega}p(x,a_{0^i{}^\frown(j+1)})$ is consistent.

Now we assume $\bigcup_{j<\omega}p(x,a_{0^i{}^\frown(2)^j{}^\frown(1)})$ is inconsistent. Similarly we can find $\psi_1(x,y)\in p(x,y)$ such that
\begin{itemize}
\item[(vi)] $\{\psi_1(x,a_\eta)\}_{\eta\in X}$ is consistent for any strict left-leaning $X$ in $\omega^{<\kappa}$,
\item[(vii)] $\{\psi_1(x,a_\eta)\}_{\eta\in X}$ is inconsistent for any strict right-veering $X$ in $\omega^{<\kappa}$.
\end{itemize}
Define a map $g$ from $\omega^{<\omega}$ to $\omega^{<\omega}$ by
\[
    g(\eta)= 
\begin{cases}
    (1) & \text{if } \;\; \eta=\emptyset\\
    f(\eta^-)^-{}^\frown(0)^\frown (2)^{t(\eta)}{}^\frown(1) & \text{if } \;\; \eta\neq\emptyset
    \end{cases}
\]
and put $c_\eta:=a_{g(\eta)}$ for each $\eta\in \omega^{<\omega}$. Then $\psi_1$ witnesses tree property with $(c_\eta)_{\eta\in\omega^{<\omega}}$ and it is a contradiction. Thus $\bigcup_{j<\omega}p(x,a_{0^i{}^\frown(2)^j{}^\frown(1)})$ is consistent.
\end{proof}
\end{lemma}

\begin{theorem}
Let $C$ be a nice graph and $G(C)$ its Mekler group. Then $\Th(C)$ is NBTP if and only if $\Th(G(C))$ is NBTP.
\begin{proof}
As in \Cref{thm: Mekler group NCTP}, it is enough to show that the implication from the left to the right. Suppose $\Th(C)$ is NBTP and $\Th(G(C))$ is BTP and fix a sufficiently saturated model $\mathbb{G}$ of $\Th(G(C))$, a transversal $X:=X^\nu{}\cup X^p\cup X^\iota$ of $\mathbb{G}$, and a subgroup $H$ of $Z(\mathbb{G})$ with $\mathbb{G}=\la X\ra\times H$. Let $\kappa$ be a cardinal such that $\kappa>2^{|T|+|\aleph_0|}$ and $\cf(\kappa)=\kappa$. By \Cref{lem: level-s-indiscernible witness of BTP} and compactness, there exist $\varphi'(\bar{x},\bar{w})$ and level-s-indiscernible $(\bar{a}'_\eta)_{\eta\in\kappa^{<\kappa}}$ such that
\begin{itemize}
\item[(i)] $\{\varphi'(\bar{x},\bar{a}'_\eta)\}_{\eta\in X}$ is consistent for any strict left-leaning path $X$ in $\kappa^{<\kappa}$,
\item[(ii)] $\{\varphi'(\bar{x},\bar{a}'_\eta)\}_{\eta\in X}$ is inconsistent for any strict right-veering path $X$ in $\kappa^{<\kappa}$,
\item[(iii)] $\{\varphi'(\bar{x},\bar{a}'_\eta)\}_{\eta\in X}$ is inconsistent for any direct sibling set $X$ in $\kappa^{<\kappa}$.
\end{itemize}
For each $\eta\in\kappa^{<\kappa}$, there exist $I_\eta\subseteq\kappa$ with $|I_\eta|=\kappa$, $(n^\nu_\eta,n^p_\eta,n^\iota_\eta,n^h_\eta)\in\omega^4$, and a tuple of terms $\bar{t}_\eta$ in the group language such that for each $i\in I_\eta$, there exists $\bar{a}''_{\eta^\frown(i)}\in X\cup H$ of the form $\bar{a}^{\nu''}_{\eta^\frown(i)}{}^\frown\bar{a}^{p''}_{\eta^\frown(i)}{}^\frown\bar{a}^{\iota''}_{\eta^\frown(i)}{}^\frown\bar{a}^{h''}_{\eta^\frown(i)}$ where $\bar{a}^{\nu''}_{\eta^\frown(i)}:=(a^{\nu''}_{\eta^\frown(i),0},...,a^{\nu''}_{\eta^\frown(i),n^\nu_\eta-1})\in X^\nu$, $\bar{a}^{p''}_{\eta^\frown(i)}:=(a^{p''}_{\eta^\frown(i),0},...,a^{p''}_{\eta^\frown(i),n_\eta^p-1})\in X^p$, $\bar{a}^{\iota''}_{\eta^\frown(i)}:=(a^{\iota''}_{\eta^\frown(i),0},...,a^{\iota''}_{\eta^\frown(i),n_\eta^\iota-1})\in X^\iota$, and $\bar{a}^{h''}_{\eta^\frown(i)}:=(a^{h''}_{\eta^\frown(i),0},...,a^{h''}_{\eta^\frown(i),n_\eta^h-1})\in H$, such that $\bar{a}'_{\eta^\frown(i)}=\bar{t}_\eta(\bar{a}''_{\eta^\frown(i)})$. By \Cref{lem: coloring on omega-kappa}, we can find $(n^\nu,n^p,n^\iota,n^h)\in\omega^4$, a tuple of terms $\bar{t}$, and a map $f$ from $\omega^{<\kappa}$ to $\kappa^{<\kappa}$ such that 
\begin{itemize}
\item[(iv)] $\bar{t}_{f(\eta)^-}\!\!=\bar{t}$ for all $\eta\in\omega^{<\kappa}$,
\item[(v)] $f(\eta)\lhd f(\eta')$ for all $\eta\lhd\eta'$,
\item[(vi)] $f(\eta)\lex f(\eta')$ for all $\eta\lex\eta'$,
\item[(vii)] $\{f(\eta^\frown(i)):i<\omega\}\subseteq \{(f(\eta^\frown(0))^-){}^\frown(i): i\in I_\eta\}$ for each $\eta\in\omega^{<\kappa}$,
\item[(viii)] for each $\eta\in\omega^{<\kappa}$, $\bar{a}''_{f(\eta)}$ is of the form $\bar{a}^{\nu''}_{f(\eta)}{}^\frown\bar{a}^{p''}_{f(\eta)}{}^\frown\bar{a}^{\iota''}_{f(\eta)}{}^\frown\bar{a}^{h''}_{f(\eta)}$, where $\bar{a}^{\nu''}_{f(\eta)}\!:=(a^{\nu''}_{f(\eta),0},...,a^{\nu''}_{f(\eta),n^\nu-1})\in X^\nu$, $\bar{a}^{p''}_{f(\eta)}\!:=(a^{p''}_{f(\eta),0},...,a^{p''}_{f(\eta),n^p-1})\in X^p$, $\bar{a}^{\iota''}_{f(\eta)}\!:=(a^{\iota''}_{f(\eta),0},...,a^{\iota''}_{f(\eta),n^\iota-1})\in X^\iota$, and $\bar{a}^{h''}_{f(\eta)}\!:=(a^{h''}_{f(\eta),0},...,a^{h''}_{f(\eta),n^h-1})\in H$, for each $\eta\in\omega^{<\kappa}$.
\end{itemize}
Thus $f(X)$ is a direct sibling set for each direct sibling set $X$. It is also easy to check that the images of strict left-leaning paths and strict right-veering paths are still strict left-leaning paths and strict right-veering paths, respectively. For each $\eta\in\omega^{<\kappa}$, let $\bar{a}_\eta:=\bar{a}''_{f(\eta)}$. Then $\varphi(\bar{x},\bar{y}):=\varphi'(\bar{x},\bar{t}(\bar{y}))$ witnesses BTP$^-$ with $(\bar{a}_\eta)_{\eta\in\omega^{<\kappa}}$. Note that, by the construction, $\bar{a}_\eta$ is of the form $\bar{a}^\nu_\eta{}^\frown\bar{a}^p_\eta{}^\frown\bar{a}^\iota_\eta{}^\frown\bar{a}^h_\eta$ where $\bar{a}^\nu_\eta\!:=(a^\nu_{\eta,0},...,a^\nu_{\eta,n^\nu-1})\in X^\nu$, $\bar{a}^p_\eta\!:=(a^p_{\eta,0},...,a^p_{\eta,n^p-1})\in X^p$, $\bar{a}^\iota_\eta\!:=(a^\iota_{\eta,0},...,a^\iota_{\eta,n^\iota-1})\in X^\iota$, and $\bar{a}^h_\eta\!:=(a^h_{\eta,0},...,a^h_{\eta,n^h-1})\in H$. We may assume $(\bar{a}_\eta)_{\eta\in\omega^{<\kappa}}$ is level-s-indiscernible by using \Cref{fact: X H are type definable} and the argument in \cite[Lemma 5.10]{KR20}. We may assume that $a^\nu_{\eta,i}\neq a^\nu_{\eta,j}$ for each $\eta\in\omega^{<\kappa}$ and $i\neq j<n^\nu$. By \Cref{fact: handle is unique} and \Cref{fact: X H are type definable}, we may assume that 
\begin{itemize}
\item[(ix)] $\{h\in X^\nu: h\text{ is a handle of }a^p_{\eta,i}\text{ for some }i<n^p\}\subseteq \bar{a}^\nu_\eta$ for each $\eta\in\omega^{<\kappa}$.
\end{itemize}

\medskip

Let $\bar{b}\models\{\varphi(\bar{x},\bar{a}_{0^i{}^\frown(1)})\}_{i<\kappa}$. Then there exist a term $\bar{s}$ and $\bar{c}:=\bar{c}^\nu{}^\frown\bar{c}^p{}^\frown\bar{c}^\iota{}^\frown\bar{c}^h$ with $\bar{c}^\nu\in X^\nu$, $\bar{c}^p\in X^p$, $\bar{c}^\iota\in X^\iota$, and $\bar{c}^h\in H$ such that $\bar{b}=\bar{t}(\bar{c})$. Note that $\bar{c}\models\{\varphi(\bar{s}(\bar{z}),a_{0^i{}^\frown(1)})\}_{i<\kappa}$ and $\varphi(\bar{s}(\bar{z}),\bar{y})$ also witnesses BTP$^-$ with $(\bar{a}_\eta)_{\eta\in\omega^{<\kappa}}$. Let $\bar{c}^p:=(c^p_0,...,c^p_{m-1})$. If $d$ is a handle of $c^p_i$ for some $i<m$, then by \Cref{fact: handle is unique} again, we may assume that $d$ satisfies one of the following.
\begin{itemize}
\item[(x)] there exists $j<n^\nu$ such that $d=a^\nu_{\eta,j}$ for all $\eta\in\omega^{<\kappa}$,
\item[(xi)] $d\neq a^\nu_{\eta,j}$ for all $\eta\in\omega^{<\kappa}$ and $j<n^\nu$.
\end{itemize}
Since $\kappa$ is sufficiently large, we may assume that
\begin{itemize}
\item[(xii)] for all $i,i'<\kappa$, 
\[
(\bar{a}^h_{0^i{}^\frown(j+1)})_{j<\omega}(\bar{a}^h_{0^i{}^\frown(2)^j{}^\frown(1)})_{j<\omega}\equiv_{\bar{c}^h}(\bar{a}^h_{0^{i'}{}^\frown(j+1)})_{j<\omega}(\bar{a}^h_{0^{i'}{}^\frown(2)^j{}^\frown(1)})_{j<\omega}
\] modulo $\Th(H)$.
\end{itemize}

Note that $X^\nu$ interprets a graph structure whose theory is $\Th(C)$, which is NBTP, and $(\bar{a}_\eta^\nu)_{\eta\in\omega^{<\kappa}}$ is still level-s-indiscernible modulo $\Th(X^\nu)$. Thus by \Cref{thm: criterion NBTP} and \Cref{fact: G=<X>H saturated}, there exist $i^*<\kappa$ and $\bar{c}^{\nu*}\in X^\nu$ such that 
\begin{itemize}
\item[(xiii)] $\bar{c}^{\nu*}\bar{a}^\nu_{0^{i*}{}^\frown(1)}$ and $\bar{c}^\nu\bar{a}^\nu_{0^{i*}{}^\frown(1)}$ have the same type modulo $\Th(X^\nu)$,
\item[(xiv)] $(\bar{a}^\nu_{0^{i*}{}^\frown(j+1)})_{j<\omega}$ or $(\bar{a}^\nu_{0^{i*}{}^\frown(2)^j{}^\frown(1)})_{j<\omega}$ is order indiscernible over $\bar{c}^{\nu*}$ modulo $\Th(X^\nu)$.
\end{itemize}

On the other hand, $\Th(H)$ is simple since it is stable, and $(\bar{a}_\eta^h)_{\eta\in\omega^{<\kappa}}$ is level-s-indiscernible modulo $\Th(H)$. Thus by \Cref{lem: simple level-s-indiscernible} and \Cref{fact: G=<X>H saturated}, there exist $i^\dagger<\kappa$ and $\bar{c}^{h\dagger}\in H$ such that
\begin{itemize}
\item[(xv)] $\bar{c}^{h\dagger}\bar{a}^h_{0^{i\dagger}{}^\frown(1)}$ and $\bar{c}^h\bar{a}^h_{0^{i\dagger}{}^\frown(1)}$ have the same type modulo $\Th(H)$,
\item[(xvi)] $(\bar{a}^h_{0^{i\dagger}{}^\frown(j+1)})_{j<\omega}$ and $(\bar{a}^h_{0^{i\dagger}{}^\frown(2)^j{}^\frown(1)})_{j<\omega}$ are order indiscernible over $\bar{c}^{h\dagger}$ modulo $\Th(H)$.
\end{itemize}
By (xii), we may assume $i^*=i^\dagger$. Put $\bar{c}^{h*}:=\bar{c}^{h\dagger}$. 

Without loss of generality, we may assume that $(\bar{a}^\nu_{0^{i*}{}^\frown(j+1)})_{j<\omega}$ is order indiscernible over $\bar{c}^{\nu*}$ modulo $\Th(X^\nu)$. By (ix) and \Cref{fact: bijection->automorphism}, there is an automorphism $\sigma$ on $\mathbb{G}$ sending $\bar{c}^\nu\bar{c}^\iota\bar{c}^h$ to $\bar{c}^{\nu*}\bar{c}^\iota\bar{c}^{h*}$ over $\bar{a}_{0^{i*}{}^\frown(1)}$. Let $\bar{c}^{p*}:=\sigma(\bar{c}^p)$ and $\bar{c}^*:=\bar{c}^{\nu*}{}^\frown\bar{c}^{p*}{}^\frown\bar{c}^\iota{}^\frown\bar{c}^{h*}$. Then $\bar{c}^{*}\models\varphi(\bar{s}(\bar{z}),\bar{a}_{0^{i*}{}^\frown(1)})$. By (x), (xi), and \Cref{fact: bijection->automorphism}, we have $\bar{a}_{0^{i*}{}^\frown(1)}\bar{c}^*\equiv\bar{a}_{0^{i*}{}^\frown(j+1)}\bar{c}^*$ for each $j<\kappa$. Thus $\bar{c}^*\models\{\varphi(\bar{s}(\bar{z}),\bar{a}_{0^{i*}{}^\frown(j+1)})\}_{j<\kappa}$. This is a contradiction.
\end{proof}
\end{theorem}

\end{document}